\documentclass{amsart}
\usepackage{color}
\usepackage{enumitem}                   
\newcommand{\Ds}{\displaystyle}
\newcommand{\calL}{{\mathcal{L}}}
\newcommand{\ph}{\varphi}

\newcommand{\terf}{s^3\theta^3\frac{x^2}{a(x)} e^{-2s\ph(x,t)}}

\def\R{\mathbb{R}}

\newtheorem{theorem}{Theorem}
\newtheorem{lemma}{Lemma}[section]

\newtheorem{remark}{Remark}[section]

\begin{document}

\title[Degenerate parabolic equations:control]{Null controllability of one dimensional degenerate parabolic equations with first order terms}
\author{J. Carmelo Flores \and Luz de Teresa}

\address{J. Carmelo Flores \newline
 Academia de Matem\'aticas  \\
Universidad  Aut\'onoma de la Ciudad  de M\'exico\\
   	Avenida la Corona 320, Col. Loma la Palma,\\
Del. Gustavo A. Madero, M\'exico D.F., 
C.P. 07160. \\
 Mexico}
\email{cflores@matem.unam.mx}

\address{Luz de Teresa \newline
Instituto de Matem\'aticas  \\
Universidad Nacional Aut\'onoma de M\'exico\\
Circuito Exterior, C.U. \\
C. P. 04510 D.F., Mexico}
\email{ldeteresa@im.unam.mx}
\subjclass[2000]{35K65, 93C20}
\keywords{degenerate parabolic systems, controllability}

\begin{abstract}
In this paper we present a null controllability result for a degenerate semilinear parabolic equation with first order terms. The main result is obtained after the proof of a new Carleman inequality for a degenerate linear parabolic equation with first order terms.
\end{abstract}
\date{}
\maketitle

\section{Introduction and main result}
\medskip
In this paper we are interested in null controllability properties of a degenerate semilinear parabolic equation. We consider 
$   a \in   C[0,1],\; a > 0 \;\textrm{in}\; (0,1],\;\; a(0)=0$. Let us fix $T>0$ and a non-empty open subset $\omega \subset (0,1)$. 
The degenerate parabolic equation we want to analyze is: 
\begin{equation}\label{e1}
    \left\{ \begin{array}{ll} y_t-  (a(x)y_{x})_{x}  +f(x,t, y, y_x)  =h1_{\omega} & \mbox{in}\;\;Q, \\
    \noalign{\smallskip}
    y(1,t)=0  \quad 
    \mbox{ and } \left\{ \begin{array}{ll} y(0,t)=0 & \mbox{for (WDP)},\\
    (ay_{x})(0,t)=0 &\mbox{ for (SDP)},
    \end{array}\right. & t \in (0,T),\\
      y(0,t)=y(1,t)=0   & t \in (0,T),\\
    y(x,0)= y_{0}(x), & \mbox{in} \;\; (0,1).
\end{array} \right.
\end{equation}
  Here, $h \in L^2(Q)$ is the control function to be determined, $1_\omega$ the characteristic function of the set  $\omega$, $y_0\in L^2(0,1)$ and $f$ is a globally Lipschitz function. The boundary conditions, weak degenerate problem (WDP), or strong degenerate problem (SDP), depend on the behavior
  of $a$ close to $x=0$. 
  
  On one hand, we  consider that the problem is weakly degenerate (WDP) if
 \begin{equation}\label{H1}
\begin{array}{ll}
\mbox{(i)} &   a \in C[0,1]\cap C^{1}(0,1],\; a > 0 \;\textrm{in}\; (0,1],\;\; a(0)=0,\\
\mbox{(ii)} &  \exists \;K \in [0,1)\; \mbox{such that}\; xa'(x)\leq Ka(x)\;\; \forall \; x \in [0,1]. 
\end{array}
\end{equation}
Here we consider Dirichlet boundary conditions $y(0)=0$. 
Notice that, under these assumptions,  $x^{K}/a(x)$ is not decreasing and then, since $0\leq K<1$,  $\frac{1}{\sqrt{a}}, \frac{1}{a} \in L^{1}(0,1)$. 

On the other hand, when the problem is strongly degenerate (SDP),  we assume
\begin{equation}\label{H2}
\begin{array}{ll}
\mbox{(i)}  &  a \in C^{1}(0,1],\; a > 0 \;\textrm{in}\; (0,1],\;\; a(0)=0,\\
\mbox{(ii)} & \exists \;K \in [1,2)\; \mbox{such that }\; xa'(x)\leq Ka(x)\;\; \forall \; x \in [0,1],\\
\mbox{(iii)} & \left\{ \begin{array}{ll}
\exists \;  \sigma \in (1,K] \; x \rightarrow \frac{a(x)}{x^{\sigma}}\; \mbox{is not decreasing close to }\; 0,  & \mbox{if} \;K>1, \\
\exists \; \sigma \in (0,1) \; x \rightarrow \frac{a(x)}{x^{\sigma}}\; \mbox{is not decreasing close to}\; 0,   & \mbox{if}  \;K=1,
\end{array}\right. 
\end{array}
\end{equation} 
the natural boundary condition at $x=0$ will be of  Neumann type:
\[ (au_{x})(0,t)= 0, \quad t \in (0,T). \]
We observe that we cannot deduce that $\frac{1}{a} \in L^1(0,1)$, however   $\frac{1}{\sqrt{a}} \in L^{1}(0,1)$, as a consequence of  ~\eqref{H2}(ii). For details, see \cite{ACF}.
 
 Our aim is to give conditions on $f$ in such a way that system \eqref{e1} is null controllable. That is,  give $H$ a Hilbert space such that for any $y^0\in H$ it exists $h\in L^2(\omega \times (0,T))$ such that the corresponding solution to \eqref{e1} satisfies
 \begin{equation}\label{cont0}
 y(T)=0.
  \end{equation}
  
In order to present our main result we need to introduce the Hilbert spaces $H_{a}^1(0,1)$ where problem is well posed:\\

\noindent WDP CASE 
\[\begin{array}{ll}
\displaystyle
 H_{a}^1(0,1) := \{ u\in L^2(0,1)  |\  u \mbox{ absolutely continuous in } \; [0,1], \\
 \displaystyle
\sqrt{a} u_{x}\in L^2(0,1)\;\mbox{and }\;
u(0)=u(1)=0 \},
\end{array}
\]
and
$$
H_{a}^{2}(0,1) := \left\{ u \in H_{a}^1(0,1)\; | \; au_{x} \in
H^{1}(0,1)\right\}.
   $$

 \noindent SDP CASE 
\[ \begin{array}{ll}
\displaystyle
 H_{a}^1(0,1) := \{ u\in L^2(0,1)  |\  u\; \mbox{locally absolutely continuous in} \; (0,1],\\ 
 \displaystyle
 \sqrt{a} u_{x} \in L^2(0,1) \;\;\mbox{ and }\;\; u(1)=0 \},
\end{array}
\]

   and

\[ \begin{array}{ll} \displaystyle
H_{a}^{2}(0,1):= \{ u \in L^2(0,1) \;\;|\;\;  u \; \mbox{locally absolutely continuous in}\; (0,1],\\
\displaystyle
\hspace{2.2cm}  au \in H^{1}_{0}(0,1),au_{x} \in
H^{1}(0,1) \;\mbox{and} \; (au_{x})(0)=0 \}.
\end{array}  \]
with norms
\[ \| u \|_{H_{a}^{1}}^{2} := \| u \|_{L^2(0,1)}^{2} + \| \sqrt{a}u_{x} \|_{L^2(0,1)}^{2}, \;\; \mbox{ and }\;\;  \| u \|_{H_{a}^{2}}^{2} := \| u \|_{H_{a}^{1}}^{2} + \| (au_{x})_{x} \|_{L^2(0,1)}^{2} \]

  In the last fifteen years the study of null controllability properties of degenerate parabolic equations has been intense. See e.g.~\cite{ACF}, \cite{CMVII}, \cite{CMVI} , \cite{FAF-C} and  \cite{MV}. In these papers the authors studied null controllability properties of linear degenerate one dimensional equations without first order terms. That is $f$ is of the form $f(x,t,y,y_x) = b(x,t) y$. In ~\cite{canfraroch}, the authors studied the null controllability properties when $f(x,t,y,y_x) = b(x,t) y+ c(x)y_x$.   
    The results we are presenting here are the generalization of the ones  presented in \cite{FdeT} in which  $a(x)=x^\alpha$ and $f$ is linear. 
  
  In order to present our result we need to introduce some hypothesis on $f$. To this end  fix a function $\beta (x)$ such that 
  \begin{equation}\label{H3}
  \frac{\beta(x)}{x}\in L^\infty (0,1).
  \end{equation} 
  
\noindent{\bf{Assumptions  A}} {\em Let  $f\colon [0,1] \times [0,T] \times \R \times \R \rightarrow \R$ be such that }
\begin{enumerate}[label=\alph*)]
\item $f(x,t;0,0)=0 \quad \forall\; (x,t) \in Q.$
\item $f(\cdot ;s,p)\in L^{\infty}(Q) \quad \forall\; (s,p)  \in \R^{2}.$
\item $f(x,t; \cdot)$ is globally Lipschitz  for every $(x,t) \in Q$ with    Lipschitz constant independent of   $(x,t)$.
\item $f(\cdot ; s,p) = g(\cdot ; s,p) + G(\cdot ; s,p)p$ for every $(s,p)  \in \R^{2}$, where
$ g(\cdot;s,p)\in L^{\infty}(Q), \; \forall (s,p)\in \R^{2} \;\; \mbox{ and } \;\;$
\[ \biggl|\frac{G(x,t;s,p)}{\beta(x)}\biggr| \leq C \quad \mbox{ almost everywhere } \;\; (x,t) \in Q,\;  \forall \; (s,p)  \in \R^{2}.\]
\end{enumerate}
  
 Our main result is:
 \begin{theorem}\label{T1}
Given  $T>0$ and $y_{0} \in L^2(0,1)$. Let us assume    ~\eqref{H1} (in the WDP) or ~\eqref{H2} (in the SDP), \eqref{H3}, and assumptions A. Then, system~\eqref{e1} is null controlable, that is, it exists $h \in
L^2(\omega \times (0,T))$ such that the solution  $y$ to~\eqref{e1} satisfies
\begin{equation}\label{null}
y(x, T) = 0 \;\;\mbox{for every }\;\; x \in [0,1].
\end{equation}
 Moreover, it exists a constant $C>0$ only depending on  $T$, such that
\[ \int_{0}^T \int_{\omega}|h|^2 dx dt \leq C \int_{0}^1 y_{0}^2(x)
dx.
\]
\end{theorem}

\begin{remark}
Observe that the assumptions on  $\beta(x)$, guarantee that
\begin{equation}\label{condicion.beta}
\frac{\beta^{2}(x)}{a(x)} \leq C \frac{x^{2}}{a(x)} \leq \frac{C}{a(1)} \;\; \mbox{a.e.}\; x \in (0,1).
\end{equation}

\end{remark} 
The aim of this paper is to prove Theorem \ref{T1}. To this end we study the linearized degenerate parabolic equation and with a fixed point argument we obtain our main result. The rest of the paper is organized as follows. In the next chapter we present  the linear problem. We prove a new Carleman inequality for the adjoint system associated with the linear one. In section 3, we prove   our main result.

\section{Linear problem}

In this section we study the null controllability of the degenerate linear equation
 \begin{equation}\label{e1-1}
    \left\{ \begin{array}{ll} y_t -  (a(x)y_{x})_{x} + b(x,t)y + \beta(x)c(x,t)y_{x}    = h1_{\omega} & \mbox{in}\;\;Q , \\
    \noalign{\smallskip}
    y(1,t)=0  \quad 
    \mbox{ and } \left\{ \begin{array}{ll} y(0,t)=0 & \mbox{for (WDP)},\\
    (ay_{x})(0,t)=0 &\mbox{for (SDP)},
    \end{array}\right. & t \in (0,T),\\
    y(x,0)= y_{0}(x), & \mbox{in} \;\; (0,1),
\end{array} \right.
\end{equation}

The following existence and uniqueness results of  solutions to \eqref{e1-1} is well known, see e.g.  ~\cite{Campiti},

\begin{theorem}\label{t.regularidad}
Suppose  that $b, c\in L^\infty ( Q)$, $a $ satisfies \eqref{H1}  (in the (SDP)) or \eqref{H2} (in the (SDP)) and $\beta$ satisfies \eqref{H3}, $h\in L^{2}(\omega \times (0,T))$, then for every $y_{0} \in L^{2}(0,1)$,
~\eqref{e1-1} has a unique solution 
\[y \in {\mathcal U}: = C([0,T]; L^2(0,1)) \cap L^2(0,T;
H_{a}^1(0,1)).\] Moreover, if  $y_{0} \in H_{a}^1(0,1)$, then
\[y \in  C([0,T];H_{a}^1(0,1)) \cap L^2(0,T;
H_{a}^{2}(0,1))\cap H^{1}(0,T; L^{2}(0,1)),\]
and it exists a positive $C_{T}$ such that 
\begin{equation}\label{d.regularidad}
\begin{array}{ll}
\|y\|_{C([0,T];H_{a}^1(0,1))}^{2} +\|y\|_{L^2(0,T;H_{a}^{2}(0,1))}^{2} +\|y\|_{H^{1}(0,T; L^{2}(0,1))}^{2}\\
   \noalign{\smallskip}
\leq \;C_{T} (\|y_{0}\|_{H_{a}^1(0,1))}^{2}  +\|h\|_{L^{2}(\omega \times (0,T)}^{2}).
\end{array}
\end{equation}
\end{theorem}
  
  It is by now well understood the null controllability properties of system \eqref{e1-1}  can be characterized in terms of the adjoint system 
  \begin{equation}\label{adjunto1}
    \left\{ \begin{array}{ll} v_t +  (a(x)v_{x})_{x}  - b(x,t)v + (\beta(x) c(x,t)v)_{x}   =0 & \mbox {in}\;\;Q, \\
    \noalign{\smallskip}
    v(1,t)=0   \quad
    \mbox{ and } \left\{ \begin{array}{ll} v(0,t)=0 & \mbox{for (WDP)},\\
    (av_{x})(0,t)=0 &\mbox{for (SDP)},
    \end{array}\right. & t \in (0,T),\\
    v(x,T)= v_{T}(x), & \mbox{in} \;\; (0,1).
\end{array} \right.
\end{equation}
  
So the aim of this section is to prove the following observability inequality:
\begin{theorem}\label{T2}  
It exists a constant $C>0$ only depending on $T>0$ such that for every solution to \eqref{adjunto1} the following holds:
\begin{equation} \label{observability}
\|v(0)\|_{L^2(0,1)}\leq C\int_0^T\!\!\int_\omega |v|^2dxdt
\end{equation}
\end{theorem}

In order to prove inequality \eqref{observability} we will need to prove a new Carleman inequality for system \eqref{adjunto1}. This is done in the next subsection.

   
\subsection{Carleman estimates for degenerate parabolic equations}
Our main result is a consequence of a Carleman inequality. We introduce some functions:
 
We define
 \begin{equation}\label{def-psi}
\psi(x) = c_{1}\biggl(c_2-\int_{0}^{x} \frac{y}{a(y)}dy   \biggr),\;\; \forall x \in [0,1] 
\end{equation}
If  $c_{1}>0$ and  $c_{2} >  \frac{1}{a(1)(2-K)}$ (where $K$ is the constant given in \eqref{H1} and \eqref{H2} (ii)  according to the WDP or the SDP), then,  
\[ \psi(x) >0, \]
for every  $x \in [0,1]$. In fact, assumptions \eqref{H1} and \eqref{H2} (ii)  implies that it exists   $K \in [0,2)$ such that \[ \frac{1}{a(\tau)} \leq \frac{1}{a(1)\tau^{K}}, \]
for every $\tau\in(0,1)$. Then, 
\[
\int_{0}^{x} \frac{\tau}{a(\tau)}\;d\tau \;\;\leq \;\; \int_{0}^{x}\frac{\tau^{1-K}}{a(1)}\;d\tau \;\;= \;\; \frac{1}{a(1)(2-K)}<c_{2}
,
\]
for every $x\in [0,1]$. 
 
For $\omega =(a,b)$  let us call $ \kappa^-=\frac{2a+b}{3}$,
$\kappa^+=\frac{a+2b}{3}$,
and let $\xi\in C^2(\mathbb{R})$ be such that $0\leq \xi\leq 1$  and
$$
\xi (x) =\begin{cases}
1 &\text{if } x\in (0,\kappa^- )\\
0 &\text{if } x\in(\kappa^+,1).
\end{cases}
$$
 
 Now, for $\omega'=(a',b')\subset\subset (\kappa^-,\kappa^+)\subset\subset \omega$ we will consider the function given in \cite{FI} for proving Carleman inequalities for a non degenerate parabolic equation. That is, we take $\rho(x)\in C^2[0,1]$ such that 
 $$\rho(x) > 0, \ x\in (0,1); \rho(0)=\rho (1)=0 \text{ and } \rho_x\not =0\text{ in } (0,a')\cup(b',1).$$
 
 We define now $\psi (x)=e^{2\lambda \|\rho\|_\infty}-e^{\lambda \rho (x)}$ and 
  $$\eta(x)=\phi(x)\xi(x)+(1-\xi(x))\psi(x).$$
 Observe that
 $$\eta'(x)=\phi'(x)\xi(x)+\phi(x)\xi'(x) -\xi'(x)\psi(x)+(1-\xi(x))\psi'(x).$$
 So for $x\in (\kappa^+, b)$,  
 $\eta'(x)=\psi'(x)\not=0$.
 We define $$\varphi (x,t)=\eta(x)\theta (t),$$ 
with
$$  \theta(t)=\frac{1}{(t(T-t))^4} \quad\forall t\in(0,T), $$
 
 Let us consider 
 \begin{equation}\label{e2-i}
    \left\{ \begin{array}{ll} v_t +  (a(x)v_{x})_{x} = F_{0} + (\beta(x) F_{1})_{x} & \mbox{in}\;\;Q , \\
    \noalign{\smallskip}
    v(1,t)=0  \quad 
    \mbox{ and } \left\{ \begin{array}{ll} v(0,t)=0 & \mbox{for the (WDP)},\\
    (av_{x})(0,t)=0 &\mbox{for the (SDP)},
    \end{array}\right. & t \in (0,T),\\
    v(x,T)= v_{T}(x), & \mbox{in} \;\; (0,1),
\end{array} \right.
\end{equation}
with  $F_{0}, F_{1} \in L^2(Q)$ and $v_{T} \in L^2(0,1)$
 
Our main result in this subsection is the following:
\begin{theorem}\label{t.Carleman} Suppose that $\omega =(a,b)$ with $a>0$. Assume that 
 ~\eqref{H1} or ~\eqref{H2}  and \eqref{H3} are verified and let  $T>0$ be given.  Then, there exists two positive constants $C$ and $s_0$, such that every solution $v$ to~\eqref{e2-i} satisfies
 \begin{equation}\label{e.Carleman}
 \begin{array}{l}
{\displaystyle \int_{0}^{T}\!\!\! \int_{0}^{1} \biggl( s\theta a(x)
v_{x}^{2} + s^{3}\theta^{3} \frac{x^{2}}{a(x)} v^{2} \biggr)
e^{-2s\varphi(x,t)}dx dt}
\\  {\displaystyle \leq \;\; C \biggl( \int_{0}^T \!\!\!\int_{\omega}e^{-2s\varphi(x,t)}v^{2} dx dt +
\int_{0}^{T}\!\!\!\int_{0}^{1} \biggl( F_{0}^{2} + s^{2} \theta^{3} \frac{\beta^{2}(x)}{a(x)} F_{1}^{2} \biggr)
e^{-2s\varphi(x,t)} dx dt\biggr)}, \end{array}
\end{equation}
for every $s\geq s_{0}$.
\end{theorem}
\medskip

The proof of Theorem~\ref{t.Carleman} is given at the end of this section as a consequence of the following result for the degenerate parabolic system:  \begin{equation}\label{e3}
    \left\{ \begin{array}{ll} v_t +  (a(x)v_{x})_{x}=F & \mbox{in}\;\;Q , \\
    \noalign{\smallskip}
    v(1,t)=0  \quad 
    \mbox{ y } \left\{ \begin{array}{ll} v(0,t)=0 & \mbox{for (WDP)},\\
    (av_{x})(0,t)=0 &\mbox{for (SDP)},
    \end{array}\right. & t \in (0,T),\\
    v(x,T)= v_{T}(x), & \mbox{in} \;\; (0,1),
\end{array} \right.
\end{equation}
where  $a$ satisfies \eqref{H1} for the (WDP) and \eqref{H2} for the (SDP), and $F \in L^{2}(Q)$. 

\begin{lemma}\label{aux.Carleman}
Let us assume that \eqref{H1} (WDP) or \eqref{H2} (SDP) hold true and let $T>0$ be given.Then, there exists two positive constants  $C$ and $s_0$, such that every solution $v$ to~\eqref{e3}, satisfies\begin{equation}\label{ineq.aux.Carleman}
 \begin{array}{ll}
 {\displaystyle \int_{0}^{T}\!\!\! \int_{0}^{1} \biggl( s\theta a(x)
v_{x}^{2} + s^{3}\theta^{3} \frac{x^{2}}{a(x)} v^{2} \biggr)
e^{-2s\varphi(x,t)}dx dt} \\
 \quad \leq \quad   C\; {\displaystyle \biggl( \int_{0}^{T}\!\!\! \int_{0}^{1} e^{-2s\varphi(x,t)} F^{2}dxdt + \int_{0}^T\!\!\! \int_{\omega}e^{-2s\varphi(x,t)}v^{2} dx dt\biggr)}
\end{array}
\end{equation}
for every $s\geq s_{0}$.
\end{lemma}

\begin{proof}
 
 We define $w(t,x)=e^{-s\varphi(t,x) } v(t,x)$.
Then, $w$ solves
$$(e^{s\varphi} w)_t+(a(x) (e^{s\varphi}w)_x)_x=F.$$
Clearly we get
$$s\varphi_tw+w_t+s(a(x) \varphi_x)_x w +s^2a(x)\varphi^2_xw+2sa(x) \varphi_xw_x+(a(x) w_x)_x=e^{-2s\varphi} F.$$
Following the ideas in \cite{CMVII} we define
$$P^+_s(w)=s\varphi_tw+s^2a(x)\varphi^2_xw+(a(x) w_x)_x$$
$$P^-_s(w)=w_t+s(a(x) \varphi_x)_x w+2sa(x) \varphi_xw_x.$$
We want to estimate the $L^2$-scalar product $(P^+_s,P^-_s)$.
We define
$$Q_1=\int_0^T\!\!\!\int_0^1 P^+_s(w)w_t ,$$
$$Q_2=2\int_0^T\!\!\!\int_0^1P^+_s(w) sa(x) \varphi_xw_x,$$
and
$$Q_3=  \int_0^T\!\!\!\int_0^1 P^+_s(w)s(a(x) \varphi_x)_x .$$

Integrating by parts, we get
$$Q_1=-\,\frac{1}{2}\,\int_0^T\!\!\!\int_0^1w^2(s\theta_{tt}\eta+2s^2a(x) \theta\theta_{t}\eta_x^2)+\left.\int_0^Taw_xw_t\right|_0^1$$
Observe that $w_t(1)=0$. To compute the boundary condition at $x=0$ we see that in the WDP case $w_t(0)=0$ and in the SDP situation we have $aw_x(0)=-s\theta a\varphi_xw(0)$. By definition of $\varphi$, $\varphi (0) = \phi(0) $. We proceed as in \cite{ACF} p. 184. Observe that $a\phi_x=-c_1 x$, $x\in (0,\tau^-]$. When $x\to 0$,  $a\phi_x\to 0$. So we don't have a contribution of the boundary term.

After integration by parts, we get,
\begin{equation}\begin{array}{l}\Ds Q_2=\left.\int_0^T s^2a\varphi_x\varphi_tw^2+s^3a^2\varphi_x^3w^2+s\varphi_x(aw_x)^2\right|_0^1\\
\noalign{\smallskip}
\Ds \phantom{\Ds Q_2=}-s^2 \int_0^T\!\!\!\int_0^1 a(\varphi_{xx}\varphi_t+\varphi_x\varphi_{xt}w^2 
-s^3\int_0^T\!\!\!\int_0^1[(a\varphi_x)_xa\varphi_x+(a\varphi_x^2)(a\varphi_x)_x]w^2\\
\noalign{\smallskip}
\Ds  \phantom{\Ds Q_2=}-s \int_0^T\!\!\!\int_0^1\varphi_{xx}(aw_x)^2.
\end{array}
\end{equation}
 We proceed again as in \cite{ACF} to estimate the boundary terms at $x=0$, taking in mind the definition of $\varphi$ close to $x=1$. It is clear that at $x=1$, $w(1)=0$. 
 In the (WDP) we get
 \begin{equation}\begin{array}{l}\Ds 
 \text{b.t} =\int_0^Ts\varphi_x(1)(a(1)w_x(1))^2 -\int_0^Ts\varphi_x(0)(a(0)w_x(0))^2\\
 \Ds \phantom{ \text{b.t}}=-s\int_0^T\theta \psi_x(1)(a(1)w_x(1))^2-s\int_0^T\theta \phi_x(0)(a(0)w_x(0))^2\\
  \Ds \phantom{ \text{b.t}}=B_1+B_2
  \end{array}
\end{equation}
By construction, $B_1\geq 0$ and since close to $x=0$, $a(x)\phi_x(x)=-c_1 x$, $B_2=0$.

In the case (SDP) we get
 $aw_x(0)=-sa\phi_x\theta w(0)$ so
 $$\int_0^T\theta \phi_x(0)(a(0)w_x(0))^2=0$$ and the term $B_1\geq 0$.
  
Similarly, 
we get
\begin{equation}
\begin{array}{l}\Ds Q_3=- \int_0^T\!\!\!\int_0^1(\varphi_t(a\varphi_x)_x +s^3 a(x)\varphi_x^2(a\varphi_x)_x]w^2\\
\Ds
-s\int_0^T\!\!\!\int_0^1[(a\varphi_x)_{xx}w+(a\varphi_x)_x w_x]aw_x+\left.\int_0^T(a\varphi_x)_x aw_xw\right|_0^1
\end{array}
\end{equation}
Proceeding as before it is easy to see that the boundary terms are zero. (See also p.184 \cite{ACF})

Altogether we get that
$$
\begin{array}{l}\displaystyle (P^+_s,P^-_s)\ge \displaystyle-\,\frac{s}{2}\,\int_0^T\!\!\!\int_0^1\theta_{tt}\eta w^2+2s^2 \int_0^T\!\!\!\int_0^1 a \theta\theta_{t}\eta_x^2w^2
 -s^2 \int_0^T\!\!\!\int_0^1 a\eta_{xx}\eta \theta \theta_tw^2
\\  \noalign{\smallskip}
\phantom{(P^+_s,P^-_s)\ge}\displaystyle-2s \int_0^T\!\!\!\int_0^1  \theta \eta_{xx}(aw_x)^2 -s^3\int_0^T\!\!\!\int_0^1[(a\varphi_x)_xa\varphi_x+(a\varphi_x^2)(a\varphi_x)_x]w^2
\\  \noalign{\smallskip}
\phantom{(P^+_s,P^-_s)\ge} \displaystyle-s \int_0^T\!\!\!\int_0^1\theta [(a\eta_x)_{xx}w+a_x\eta_xw_x]aw_x-s^3\int_0^T\!\!\!\int_0^1\theta^3a^2\eta_x^2\eta_{xx}w^2\\  \noalign{\smallskip}\Ds \phantom{(P^+_s,P^-_s)\ge}+\int_0^T\!\!\!\int_0^1\theta_t\theta\eta(a\eta_x)_xw^2
\end{array}
$$

We can decompose the right hand side of this inequality as the sum of integral over $[0,\kappa^-]$  , $[\kappa^-,\kappa^+]$ and $[\kappa^+,1]$. Observe that our choice of weight functions allows to reason on $[0,\kappa^-]$ similarly to the derivation of (3.10) in \cite{ACF}. On the other hand, on  $(\kappa^+, 1)$ we obtain classical Carleman estimates with terms on the gradient and the function over the set $\widetilde \omega$. Proceeding as in \cite{ACF} to bound the terms over $(0,\kappa^-)$ and as traditional Carleman estimates (see e.g. \cite{FI}) for the terms over $(\kappa^+, 1)$ and having in mind that $a(x)$ is bounded bellow by a positive constant in $(\kappa^+,1)$ we obtain, with appropriate constants:
\begin{equation*} 
\begin{aligned}
&\iint_Q\biggl( s\theta a(x)
v_{x}^{2} + s^{3}\theta^{3} \frac{x^{2}}{a(x)} v^{2} \biggr)e^{-2s\varphi}\,dx\,dt   \\
&\leq C\Big( \iint_Q
e^{-2s\varphi }F^2\,dx\,dt+\int_0^T\!\!\!\int_{\widetilde\omega} e^{-2s\varphi}v^2+\int_0^T\!\!\!\int_{\widetilde\omega} e^{-2s\varphi}v_x^2\,dx\,dt \Big).
\end{aligned}
\end{equation*}
To get \eqref{ineq.aux.Carleman} we eliminate the term $\int_0^T\!\!\!\int_{\widetilde\omega} e^{-2s\varphi}v_x^2\,dx\,dt$ performing local energy estimates and ``growing'' $\widetilde \omega$ to $\omega$. That is, we use Cacciopoli's inequality, which is:
$$ \int_0^T\!\!\!\int_{\widetilde\omega} e^{-2s\varphi}v_x^2\,dx\,dt\leq C\left(\int_0^T\!\!\!\int_{ \omega} e^{-2s\varphi}v^2dxdt+\int_0^T\!\!\!\int_{ \Omega} e^{-2s\varphi}F^2dxdt\right).$$
This completes the proof.

\end{proof}

\subsection{Proof of Theorem~\ref{t.Carleman}}
 \bigskip
 
Here we use the technique introduced in~\cite{yamamoto}; that is, we use Carleman inequality~\eqref{ineq.aux.Carleman} to obtain two null controlability auxiliary results that are used in the proof of   Carleman inequality~\eqref{e.Carleman}. 
\medskip

We consider the following systems:

\begin{equation}\label{auxcon1}
    \left\{ \begin{array}{ll} z_t -  (a(x)z_{x})_{x} = s^{3} \theta^{3} {\displaystyle \frac{x^2}{a(x)}} e^{-2s\varphi(x,t)}f + u1_{\omega} & \mbox{en}\;\;Q , \\
    \noalign{\smallskip}
    z(1,t)=0 \quad
    \mbox{ and } \left\{ \begin{array}{ll} z(0,t)=0 & \mbox{for the (WDP)},\\
    (az_{x})(0,t)=0 &\mbox{for the (SDP)},
    \end{array}\right. & t \in (0,T),\\
    z(x,0)= 0, & \mbox{in} \;\; (0,1)
\end{array} \right.
\end{equation}

and

\begin{equation}\label{auxcon2}
    \left\{ \begin{array}{ll} z_t -  (a(x)z_{x})_{x} = s \theta (e^{-2s\varphi(x,t)}\sqrt{a}f)_{x} + u1_{\omega} & \mbox{in}\;\;Q , \\
    \noalign{\smallskip}
    z(1,t)=0 \quad
    \mbox{ and } \left\{ \begin{array}{ll} z(0,t)=0 & \mbox{for the (WDP)},\\
    (az_{x})(0,t)=0 &\mbox{for the (SDP)},
    \end{array}\right. & t \in (0,T),\\
    z(x,0)= 0, & \mbox{in} \;\; (0,1),
\end{array} \right.
\end{equation}
with $f \in L^{2}(Q)$.
\medskip

We define $P_{w} = \{ p\in C^2(\bar Q) | \;p(0)=0, p(1)=0 \}$ for the (WDP) case  and $P_{S} = \{ p\in C^2(\bar Q) | \;(ap_{x}
)(0)=0, p(1)=0 \}$ for the (SDP) case .  Let $\calL$ and $\calL^{*}$ be two linear operators defined as:
\[ \calL p=p_t-(a(x)p_x)_x \]
and
\[\calL^*p=p_t+(a(x)p_x)_x,\]
for every  $p$ in $P_{w}$ for the (WDP) case  and $p\in P_{S}$  for the (SDP) case. With this, we define
 \[\lambda (p,p') = \int_{0}^{T}\!\!\! \int_{0}^{1}e^{-2s\varphi(x,t)} \calL^*p\calL^*p'dxdt 
 +\int_{0}^{T}\!\!\!\int_{\omega}e^{-2s\varphi(x,t)} 
 pp'dxdt,\]
for every $p,p'$ in $P_{w}$ for the (SDP) case  and for every  $p,p'$ in $P_{S}$ for the (SDP) case. It is not difficult to check that   $\lambda (\cdot,\cdot)$ is a bilinear, positive and symmetric form. Then, $\lambda(\cdot,\cdot)$ defines an internal product in $P_{w}$ and in  $P_{S}$. We define $\mathcal{P}_{w}$ 
 as the closure of  $P_{w}$ with the norm $\| p\|_{\mathcal{P}_{w}}=(\lambda(p,p))^{1/2}$ in the case (WDP) and $\mathcal{P}_{S}$ 
 as the closure of  $P_{S}$  with the norm $\| p\|_{\mathcal{P}_{w}}=(\lambda(p,p))^{1/2}$ in the case (SDP).
\medskip

As a consequence of  Carleman inequality~\eqref{ineq.aux.Carleman}, we get the following null controllability result.

\begin{theorem}\label{cont.cero.aux}
Let $T>0$ and $f \in L^{2}(Q)$ be given. Then, we have that:
\begin{enumerate}
\item [\em{(1)}]For system~\eqref{auxcon1}, it exists a control $u$ and a state  $z$,  such that $z(x,T) = 0$ in $(0,1)$ and
\begin{equation}\label{auxcon1.0}
\begin{array}{ll}
{\displaystyle \int_{0}^{T}\!\!\! \int_{0}^{1} e^{2s\varphi(x,t)} z^{2} dxdt
+ \int_{0}^{T}\!\!\! \int_{\omega} e^{2s\varphi(x,t)} u^{2} dxdt}\\
{\displaystyle \;\leq \;C\; \int_{0}^{T}\!\!\! \int_{0}^{1}s^{3} \theta^{3}
\frac{x^{2}}{a(x)} e^{-2s\varphi(x,t)}f^{2}dxdt.}
\end{array}
\end{equation}
is satisfied.
\item [\em{(2)}]For system~\eqref{auxcon2}, it exists a   control $u$  and a state $z$, such that $z(x,T) = 0$ in $(0,1)$ and
the following is satisfied
\begin{equation}\label{auxcon2.0}
\begin{array}{ll}
{\displaystyle \int_{0}^{T}\!\!\! \int_{0}^{1} e^{2s\varphi(x,t)} z^{2} dxdt
+ \int_{0}^{T} \!\!\!\int_{\omega} e^{2s\varphi(x,t)} u^{2} dxdt}\\
{\displaystyle \;\leq \;C\;  \int_{0}^{T}\!\!\! \int_{0}^{1} s \theta
 e^{-2s\varphi(x,t)}f^{2}dxdt.}
\end{array}
\end{equation}
\end{enumerate} 
\end{theorem}
\begin{proof}

Given $f\in L^2(Q)$ y $T>0$
we consider the following problem:

\begin{equation}\label{prob}
\left\{ 
\begin{array}{l}\Ds
\calL( e^{-2s\varphi(x,t)}\calL^* p)-\terf f =  e^{-2s\varphi(x,t)} p\chi_\omega \;\; \text{ en }Q \\  \Ds
p(1)=e^{-2s\varphi(x,t)}\calL^*p(1)=0 \text{ and } \left\{\begin{array}{l} 
\Ds
p(0)=e^{-2s\varphi(x,t)}\calL^*p(0)=0 \\ \mbox{ for the WDP, }\\ 
\Ds
ap_x(0)=e^{-2s\varphi(x,t)}\calL^*ap_x(0)=0 \\ \mbox{ for the SDP, }
\end{array}\right.\text{ in } (0,T)\\ 
\Ds
e^{-2s\varphi(x,t)} \calL^* p(x,0)=e^{-2s\varphi(x,t)} \calL^* p(x,T)=0 \;\;\mbox{in} \;\; (0,1).
\end{array}
\right.  
\end{equation}
\medskip

We show that has a solution in the case WDP  (the argument in the SDP is analougous). 

Carleman inequality~\eqref{ineq.aux.Carleman} implies
\[
\begin{array}{ll}
{\displaystyle \int_{0}^{T}\!\!\! \int_{0}^{1}\biggl(s\theta a(x)p_x^2 + s^3\theta^3\frac{x^2}{a(x)} p^2\biggr)e^{-2s\ph(x,t)}
dxdt 
}\\
{\displaystyle \le\;\; C\biggl( \int_{0}^{T}\!\!\! \int_{0}^{1}e^{-2s\varphi(x,t)} |\calL^*p|^2\;dxdt +\int_{0}^{T}\!\!\!\int_{\omega}e^{-2s\varphi(x,t)} p^2dxdt\biggr),}
\end{array}
\]
for every $p\in \mathcal P_{w}$.
In consequence, for every  $p \in \mathcal P_{w}$

$$ \int_{0}^{T}\!\!\! \int_{0}^{1}\biggl(s\theta a(x)p_x^2 + s^3\theta^3\frac{x^2}{a(x)} p^2\biggr)e^{-2s\ph(x,t)}dxdt<\infty.$$
On the other hand,
\begin{eqnarray*}\ell (p)&=&  - \int_{0}^{T}\!\!\! \int_{0}^{1}\terf fpdxdt   \\
&\le & C\biggl( \int_{0}^{T}\!\!\! \int_{0}^{1} \terf f^2\biggr)^{1/2}\|p\|_{	\mathcal P_{w}}
\end{eqnarray*}
and then $\ell$ es linear and continuous on $\mathcal P_{w}$.
From  Lax-Milgram Theorem it exists a unique   $\bar p\in \mathcal P$ solution to the problem
$$\lambda (\bar p, p')=\ell (p') \quad \forall p'\in \ell (p').$$ 
In consequence, $\bar p$ is a  weak solution of    \eqref{prob}. 
It can be shown that if $p$ is a classical solution, then it satisfies $e^{-2s\varphi(x,t)}\calL^*p(0)=0 $, $e^{-2s\varphi(x,t)}\calL^*p(1)=0 $. Observe that $p(0) =0=p(1)$  in  $\mathcal P_w$ and  $e^{-2s\varphi(x,t)}\calL^*p(x,0)=0 $, $e^{-2s\varphi(x,t)}\calL^*p(x,T)=0$. In fact, multiplying~\eqref{prob} by $p'\in \mathcal P_{w}$ and integrating by parts we get
\[ \int_{0}^ {1}e^{-2s\varphi(x,t)} \calL^*p p'\biggr|_{0}^{T}\;dx + \int_{0}^{T}a(x) e^{-2s\varphi(x,t)} \calL^*p(x,t) p'_{x}\biggr|_{0}^{1}\; dt +\ell (p') = \lambda(p,p'), \;\;\; \forall p'\in \mathcal P_{d}.\]
Taking an appropriate $p'\in \mathcal P_{d}$ we conclude that  $e^{-2s\varphi(x,t)} \calL^*p(x,t)= 0 \;\;\mbox{in} \;\partial Q$. 

We define now

\begin{equation}\label{2}
\bar z= e^{-2s\varphi(x,t)} \calL^*\bar p; \ \bar u = e^{-2s\varphi(x,t)} \bar p\chi_\omega
\end{equation}
Then, $\bar z$ solves
\begin{equation}\label{prob2}
\left\{ 
\begin{array}{l}\Ds
\bar z_t -(a(x)\bar z_x)_x=\terf f+  \bar u\chi_\omega \text{ in }Q \\  \Ds
\bar z(1) =0 \text{ y }  \bar z(0)=0 \text{ WDP } \\ \Ds
\bar z(x,0)=\bar z(x,T)=0
\end{array}
\right.  
\end{equation}

From \eqref{2}, 
\begin{equation}\label{4}
\bar p_{t}+(a(x)\bar p_x)_x= e^{2s\varphi(x,t)}\bar z.
\end{equation}
In order to estimate the norms of   $\bar z$ and of the control $\bar u$ we multiply  \eqref{4} by $\bar z$.
Then,
 $$ \int_{0}^{T}\!\!\! \int_{0}^{1} (\bar p_t+(a(x)\bar p_x)_x)\bar zdxdt=  \int_{0}^{T}\!\!\! \int_{0}^{1}e^{2s\varphi(x,t)}\bar z
 ^2\;dxdt$$
 Integrating by parts in space and time and using  \eqref{prob2} we see that 
$$- \int_{0}^{T}\!\!\! \int_{0}^{1} \terf f \bar p\;dxdt -\int_{0}^{T}\!\!\! \int_{\omega} \bar u\bar p\;dxdt =  \int_{0}^{T}\!\!\! \int_{0}^{1} e^{2s\varphi(x,t)}\bar z^2\;dxdt$$
$$I_1+I_2=  \int_{0}^{T}\!\!\! \int_{0}^{1} e^{2s\varphi(x,t)}\bar z^2\;dxdt$$
We know that  $\bar u=e^{-2s\varphi(x,t)} \bar p\chi_\omega$,
then
$$I_2= -\int_{0}^{T} \int_{\omega} e^{-2s\varphi(x,t)}\bar p^2\;dxdt.$$
On the other hand,
$$I_1\leq \frac{1}{2\delta}  \int_{0}^{T}\!\!\! \int_{0}^{1} \terf f ^2\;dxdt+ \frac{\delta}{2}  \int_{0}^{T}\!\!\! \int_{0}^{1} 
\terf \bar p^2\;dxdt.$$ 
Since  $\bar p$ is solution to  \eqref{4}  we can apply Carleman inequality ~\eqref{ineq.aux.Carleman} with right hand side  $e^{2s\varphi(x,t)}\bar z$. 
We have that
\begin{eqnarray*}
\frac{\delta}{2}  \int_{0}^{T}\!\!\! \int_{0}^{1}\terf \bar p^2dxdt &\leq & \frac{C\delta}{2}
\biggl( \int_{0}^{T}\!\!\! \int_{0}^{1} e^{-2s\varphi(x,t)}e^{4s\varphi(x,t)} \bar z^2dxdt\\
 & + & \int_{0}^{T}\!\!\!  \int_{\omega}e^{-2s\varphi(x,t)}   \bar p^2\;dxdt\biggr)
\end{eqnarray*}
From \eqref{2} we get
$$\int_{0}^{T} \!\!\! \int_{\omega}e^{-2s\varphi(x,t)}\bar p^2\;dxdt= \int_{0}^{T}\!\!\! \int_\omega e^{2s\varphi(x,t)}\bar u^2\;dxdt$$
In conclusion, we got 
$$ \int_{0}^{T} \!\!\! \int_{\omega} e^{2s\varphi(x,t)}\bar u^2\;dxdt+  \int_{0}^{T}\!\!\! \int_{0}^{1} e^{2s\varphi(x,t)}\bar z^2\;dxdt\leq C   \int_{0}^{T}\!\!\! \int_{0}^{1}\terf f ^2\;dxdt$$
With this we conclude the proof of \eqref{auxcon1.0}. The proof of \eqref{auxcon2.0}  is similar, we only need to consider functional 
\begin{eqnarray*}\ell (p)&=&  \int_{0}^{T}\!\!\! \int_{0}^{1}s\theta e^{-2s\varphi(x,t)}\sqrt{a(x)}fp_{x}dxdt.  
\end{eqnarray*}

\end{proof}

\begin{proof}[Proof of Theorem~\ref{t.Carleman}]
The proof is given in two steps:
\medskip

\noindent \emph{Step 1. Two auxiliary null controllability problems}

\noindent  We apply the previous result to $v \in L^2(Q)$ solution to~\eqref{adjunto1}. We deduce the existence of a   control $\hat{v}$ and a state $\hat{z}$
such that
\begin{equation}\label{auxcon1.1}
    \left\{ \begin{array}{ll} \hat{z}_t -  (a(x)\hat{z}_{x})_{x} = s^{3} \theta^{3} \frac{x^2}{a(x)} e^{-2s\varphi(x,t)}v + \hat{v}1_{\omega} & \mbox{in}\;\;Q , \\
    \noalign{\smallskip}
    \hat{z}(1,t)=0  \quad
    \mbox{ and } \left\{ \begin{array}{ll} \hat{z}(0,t)=0 & \mbox{WDP} \\
    (a\hat{z}_{x})(0,t)=0 &\mbox{SDP},
    \end{array}\right. & t \in (0,T),\\
    \hat{z}(x,0)=\hat{z}(x,T)=0, & \mbox{in} \;\; (0,1),
\end{array} \right.
\end{equation}
and
\begin{equation}\label{auxcon1.2}
\begin{array}{ll}
{\displaystyle  \int_{0}^{T}\!\!\! \int_{0}^{1} e^{2s\varphi(x,t)}
{\hat{z}}^{2} dxdt
+ \int_{0}^{T} \!\!\! \int_{\omega} e^{2s\varphi(x,t)} {\hat{v}}^{2} dxdt}
\\
{\displaystyle \;\leq \;C\; \int_{0}^{T}\!\!\!  \int_{0}^1 s^{3} \theta^{3}
\frac{x^2}{a(x)} e^{-2s\varphi(x,t)}v^{2}dxdt.}
\end{array}
\end{equation}

If we multiply by $s^{-2} \theta^{-3} e^{2s\varphi(x,t)} \hat{z}$
  equation  \eqref{auxcon1.2} and we integrate by parts, we conclude that
\[
\begin{array}{ll}
{\displaystyle  \int_{0}^{T}\!\!\! \int_{0}^{1} s^{-2}\theta^{-3} a(x)
e^{2s\varphi(x,t)} {\hat{z}_{x}}^{2} dxdt \;=\; -  \int_{0}^{T}\!\!\! \int_{0}^{1} s^{-2}\theta^{-3} 
e^{2s\varphi(x,t)} \hat{z} \hat{z}_{t}dxdt} 
\end{array}
\]
\[\begin{array}{ll}
{\displaystyle -\;\;  \int_{0}^{T}\!\!\! \int_{0}^{1} s^{-2}\theta^{-3} a(x)
(e^{2s\varphi(x,t)})_{x} \hat{z}_{x}\hat{z} dxdt +  \int_{0}^{T}\!\!\! \int_{0}^{1} s \frac{x^2}{a(x)} v\hat{z}_{x} dxdt} \\
{\displaystyle +\;\; \int_{0}^{T}\!\!\!  \int_{\omega} s^{-2}\theta^{-3} 
e^{2s\varphi(x,t)} \hat{v}\hat{z} dxdt}=H_{1} +H_{2}+H_{3}+H_{4}.
\end{array}
\]
We observe that for every $(x,t) \in Q_1$, $|(\theta^{-3}e^{2s\varphi(x,t)})_{t}| \leq C s e^{2s\varphi(x,t)}$ and $|(e^{-2s\varphi(x,t)})_{x}| \leq C s \theta e^{2s\varphi(x,t)}$, and for every  $x \in (0,1)$, $x^{2}/a(x) \leq C$. Then, applying  \eqref{auxcon1.2}, we got
\begin{eqnarray*}
 H_{1} &=&  -\frac{1}{2}  \int_{0}^{T}\!\!\! \int_{0}^{1} s^{-2}\theta^{-3} 
e^{2s\varphi(x,t)}\frac{d}{dt}[{\hat{z}}^{2}] dxdt 
 \; \leq \;  \frac{1}{2} \int_{0}^{T}\!\!\! \int_{0}^{1} s^{-2}|\theta^{-3} 
e^{2s\varphi(x,t)})_{t}| {\hat{z}}^{2}dxdt 
\end{eqnarray*}
\begin{eqnarray}\label{auxcon1.3}
 & \leq & C \;  \int_{0}^{T}\!\!\! \int_{0}^{1} s^{-1} 
e^{2s\varphi(x,t)} {\hat{z}}^{2}dxdt
 \; \leq \; C\;   \int_{0}^{T}\!\!\! \int_{0}^{1} s^{3}\theta^{3} \frac{x^2}{a(x)} e^{-2s\varphi(x,t)} v^{2} dxdt
\end{eqnarray}           

and

\begin{eqnarray}\label{auxcon1.4}
 H_{2}  &\leq &    C\;   \int_{0}^{T}\!\!\! \int_{0}^{1} s^{-1}\theta^{-2} a(x) e^{2s\varphi(x,t)} |\hat{z_{x}}||\hat{z}| dxdt \nonumber \\
 &=& C \;  \int_{0}^{T}\!\!\! \int_{0}^{1} s^{-1}\theta^{-3/2} a(x) e^{2s\varphi(x,t)}\theta^{-1/2}|\hat{z_{x}}||\hat{z}| dxdt  \\
  &\leq &   \frac{1}{2}\;  \int_{0}^{T}\!\!\! \int_{0}^{1} s^{-2}\theta^{-3} a(x) e^{2s\varphi(x,t)}{\hat{z_{x}}}^{2} dxdt  + C  \int_{0}^{T}\!\!\! \int_{0}^{1} e^{2s\varphi(x,t)}{\hat{z}}^{2} dxdt \nonumber\\
\end{eqnarray}

\begin{eqnarray*}  
  &\leq &  \frac{1}{2}\;  \int_{0}^{T}\!\!\! \int_{0}^{1} s^{-2}\theta^{-3} a(x) e^{2s\varphi(x,t)}{\hat{z_{x}}}^{2} dxdt  
  +  C  \int_{0}^{T}\!\!\! \int_{0}^{1} s^{3}\theta^{3} \frac{x^2}{a(x)} e^{-2s\varphi(x,t)} v^{2} dxdt \nonumber .
\end{eqnarray*}

Assuming $s_{0} \geq 1$, we have that
\begin{equation}\label{auxcon1.5}
\begin{array}{lll}
 H_{3} &\leq &{\displaystyle    \int_{0}^{T}\!\!\! \int_{0}^{1} s^{3/2} \frac{x^2}{a(x)}
v\hat{z} dxdt}\\ 
 &\leq & {\displaystyle C\;\biggl( \int_{0}^{T}\!\!\! \int_{0}^{1} s^{3}\theta^{3}  \frac{x^2}{a(x)}
e^{-2s\varphi(x,t)} v^{2}dxdt\biggr)^{1/2}  \biggl(  \int_{0}^{T}\!\!\! \int_{0}^{1}  
e^{2s\varphi(x,t)} {\hat{z}}^{2}dxdt\biggr)^{1/2}  }\\
&\leq & C\; {\displaystyle  \int_{0}^{T}\!\!\! \int_{0}^{1} s^{3}\theta^{3} \frac{x^2}{a(x)} e^{-2s\varphi(x,t)} v^{2} dxdt}
\end{array}
\end{equation}           
and
\begin{equation}\label{auxcon1.6}
\begin{array}{lll}
 H_{4} &\leq &{\displaystyle   \int_{0}^{T}\!\!\! \int_\omega e^{2s\varphi(x,t)}
|\hat{v}\hat{z}| dxdt} \\ 
 &\leq & {\displaystyle C\;\biggl( \int_{0}^{T}\!\!\! \int_{0}^{1} e^{2s\varphi(x,t)} {\hat{v}}^{2}dxdt\biggr)^{1/2}  \biggl(  \int_{0}^{T}\!\!\! \int_{0}^{1}  
e^{2s\varphi(x,t)} {\hat{z}}^{2}dxdt\biggr)^{1/2} }\\
&\leq & C\; {\displaystyle  \int_{0}^{T}\!\!\! \int_{0}^{1} s^{3}\theta^{3} \frac{x^2}{a(x)} e^{-2s\varphi(x,t)} v^{2} dxdt}
\end{array}
\end{equation}  
Applying \eqref{auxcon1.3}, \eqref{auxcon1.4}, \eqref{auxcon1.5} and \eqref{auxcon1.6}, we conclude
\[
\begin{array}{ll}
{\displaystyle  \int_{0}^{T}\!\!\! \int_{0}^{1} s^{-2}\theta^{-3} a(x)
e^{2s\varphi(x,t)} {\hat{z}_{x}}^{2} dxdt \leq C  \int_{0}^{T}\!\!\! \int_{0}^{1} s^{3}\theta^{3} \frac{x^2}{a(x)} 
e^{-2s\varphi(x,t)} v^{2} dxdt,}
\end{array}
\]
which combined with \eqref{auxcon1.2} gives
\begin{equation}\label{auxcon1.7}
\begin{array}{ll}
{\displaystyle  \int_{0}^{T}\!\!\! \int_{0}^{1} e^{2s\varphi(x,t)}
{\hat{z}}^{2} dxdt + \int_{0}^{T}\!\!\! \int_\omega e^{2s\varphi(x,t)}
{\hat{v}}^{2} dxdt} \\ {\displaystyle + \int_{0}^{T}\!\!\! \int_{0}^{1}
s^{-2}\theta^{-3} a(x) e^{2s\varphi(x,t)} {\hat{z}_{x}}^{2} dxdt
 \;\leq \;C \int_{0}^{T}\!\!\! \int_0^1 s^{3} \theta^{3} \frac{x^{2}}{a(x)}
e^{-2s\varphi(x,t)}v^{2}dxdt}
\end{array}
\end{equation}
for every $s \geq s_{0}$.
\bigskip

On the other hand, applying part {\em 2} of  Theorem~\ref{cont.cero.aux} for $f=\sqrt{a}v_{x} \in L^2(Q)$, where $v$ is the solution to~\eqref{adjunto1}, we can deduce the existence of a control  $\tilde{v}$ and a state   $\tilde{z}$
such that
 \begin{equation}\label{auxcon2.1}
    \left\{ \begin{array}{ll} \tilde{z}_t -  (a(x)\tilde{z}_{x})_{x} = s \theta (e^{-2s\varphi(x,t)}a(x)v_{x})_{x} + \tilde{v}1_{\omega} & \mbox{in}\;\;Q , \\
    \noalign{\smallskip}
    \tilde{z}(1,t)=0  \quad
    \mbox{ and } \left\{ \begin{array}{ll} \tilde{z}(0,t)=0 & \mbox{WDP} ,\\
    (a\tilde{z}_{x})(0,t)=0 &\mbox{SDP},
    \end{array}\right. & t \in (0,T),\\
    \tilde{z}(x,0)=\tilde{z}(x,T)=0, & \mbox{in} \;\; (0,1),
\end{array} \right.
\end{equation}
and
\begin{equation}\label{auxcon2.2}
\begin{array}{ll}
{\displaystyle  \int_{0}^{T}\!\!\! \int_{0}^{1} e^{2s\varphi(x,t)}
{\tilde{z}}^{2} dxdt
+ \int_{0}^{T}\!\!\! \int_{\omega} e^{2s\varphi(x,t)} {\tilde{v}}^{2} dxdt}\\
 \leq C{\displaystyle \; \int_{0}^{T} \!\!\!\int_{0}^1 s \theta a(x)
e^{-2s\varphi(x,t)}v_{x}^{2}dxdt}
\end{array}
\end{equation}
for every  $s \geq s_{0}$.
\medskip

It is not difficult to see that 
\[
\begin{array}{ll}
{\displaystyle  \int_{0}^{T}\!\!\! \int_{0}^{1} s^{-2}\theta^{-2} a(x)
e^{2s\varphi(x,t)} {\tilde{z}_{x}}^{2} dxdt \;=\; -  \int_{0}^{T}\!\!\! \int_{0}^{1} s^{-2}\theta^{-2} 
e^{2s\varphi(x,t)} \tilde{z} \tilde{z}_{t}dxdt} \\
\noalign{\smallskip}
{\displaystyle -\;\;  \int_{0}^{T}\!\!\! \int_{0}^{1} s^{-2}\theta^{-2} a(x)
(e^{2s\varphi(x,t)})_{x} \tilde{z}_{x}\tilde{z} dxdt -  \int_{0}^{T}\!\!\! \int_{0}^{1} s^{2}\theta^{-1} a(x)
v_{x} \tilde{z}_{x} dxdt} \\
\noalign{\smallskip}
{\displaystyle -\;\; \int_{0}^{T}\!\!\! \int_{0}^{1} s^{-1}\theta^{-1} a(x)
(e^{2s\varphi(x,t)})_{x} e^{-2s\varphi(x,t)}v_{x}\tilde{z} dxdt + \int_{0}^{T}\!\!\! \int_{\omega} s^{-2}\theta^{-2} 
e^{2s\varphi(x,t)} \tilde{v}\tilde{z} dxdt}\\
\noalign{\smallskip}
=L_{1} +L_{2}+L_{3}+L_{4}+L_{5}.
\end{array}
\]
Proceeding as before, we conclude that 
\[
\begin{array}{ll}
{\displaystyle  \int_{0}^{T}\!\!\! \int_{0}^{1} s^{-2}\theta^{-2} a(x)
e^{2s\varphi(x,t)} {\tilde{z}_{x}}^{2} dxdt \leq C  \int_{0}^{T}\!\!\! \int_{0}^{1} s \theta a(x) 
e^{-2s\varphi(x,t)} v_{x}^{2} dxdt,}
\end{array}
\]
combined with inequality  \eqref{auxcon2.2} gives

\begin{equation}\label{auxcon2.8}
\begin{array}{ll}
{\displaystyle  \int_{0}^{T}\!\!\! \int_{0}^{1} e^{2s\varphi(x,t)}
{\tilde{z}}^{2} dxdt
+ \int_{0}^{T}\!\!\! \int_{\omega} e^{2s\varphi(x,t)} {\tilde{v}}^{2} dxdt}\\
\noalign{\smallskip}
{\displaystyle \qquad+ \int_{0}^{T}\!\!\!\int_{0}^{1} s^{-2}
\theta^{-2}a(x)e^{2s\varphi(x,t)}{\tilde{z}_{x}}^{2}dxdt \;\leq
\;C\; \int_{0}^{T} \!\!\!\int_{0}^1 s \theta a(x)
e^{-2s\varphi(x,t)}v_{x}^{2}dxdt}
\end{array}
\end{equation}
for every$s \geq s_{0}$.
\bigskip

\noindent \emph{Step 2. Proof of Carleman inequality \eqref{e.Carleman}}
\newline We multiply \eqref{auxcon1.1} by $v$ solution to~\eqref{adjunto1}. Then, integrating by parts and using H\"older's inequality, we obtain
\[
\begin{array}{l}
 {\displaystyle  \int_{0}^{T}\!\!\! \int_0^1 s^{3} \theta^{3} \frac{x^{2}}{a(x)}
e^{-2s\varphi(x,t)} v^{2} dx dt = - \int_{0}^{T}\!\!\! \int_0^1\hat{z}v_{t}dxdt - \int_{0}^{T}\!\!\! \int_0^1\hat{z}(a(x)v_{x})_{x}dxdt}\\
{\displaystyle - \int_{0}^{T}\!\!\! \int_\omega\hat{v}vdxdt =- \int_{0}^{T}\!\!\! \int_0^1F_{0}\hat{z}dxdt + \int_{0}^{T}\!\!\! \int_0^1\beta(x)F_{1}\hat{z}_{x}dxdt- \int_{0}^{T}\!\!\! \int_\omega\hat{v}vdxdt \leq}
\end{array}
\]
\[\begin{array}{l}
{\displaystyle \biggl(  \int_{0}^{T}\!\!\! \int_{0}^{1} e^{-2s\varphi(x,t)} F_{0}^{2} dxdt  
  + \int_{0}^{T}\!\!\! \int_0^1 s^{2} \theta^{3} \frac{\beta^{2}(x)}{a(x)}
e^{-2s\varphi(x,t)}F_{1}^{2}dxdt  + \int_{0}^{T}\!\!\! \int_\omega
e^{-2s\varphi(x,t)} v^{2}dxdt\biggr)^{1/2} \times}\\
  {\displaystyle \biggl(  \int_{0}^{T}\!\!\! \int_{0}^{1} e^{2s\varphi(x,t)}
{\hat{z}}^{2} dxdt + \int_{0}^{T}\!\!\! \int_0^1 s^{-2} \theta^{-3}
a(x) e^{2s\varphi(x,t)}{\hat{z}_{x}}^{2}dxdt
 + \int_{0}^{T}\!\!\!
\int_{\omega} e^{2s\varphi(x,t)}\hat{v}^{2}dxdt\biggr)^{1/2}.}
\end{array}
\]
From \eqref{auxcon1.7}, we deduce
\begin{equation}\label{paso2.1}
\begin{array}{ll}
{\displaystyle  \int_{0}^{T}\!\!\! \int_0^1 s^{3} \theta^{3}
\frac{x^{2}}{a(x)} e^{-2s\varphi(x,t)} v^{2} dx dt \leq C
\biggl(  \int_{0}^{T}\!\!\! \int_{0}^{1} e^{-2s\varphi(x,t)} F_{0}^{2} dxdt} \\
\noalign{\smallskip}
 {\displaystyle + \int_{0}^{T}\!\!\! \int_0^1 s^{2} \theta^{3} \frac{\beta^{2}(x)}{a(x)}
e^{-2s\varphi(x,t)}F_{1}^{2}dxdt  + \int_{0}^{T}\!\!\! \int_\omega
e^{-2s\varphi(x,t)} v^{2}dxdt\biggr).}
\end{array}
\end{equation}
\medskip

Analogously, if we multiply by  $v$ system \eqref{auxcon2.1} and integrating by parts, we conclude
\[
\begin{array}{l}
 {\displaystyle  \int_{0}^{T}\!\!\! \int_0^1 s \theta a(x)
e^{-2s\varphi(x,t)} v_{x}^{2} dx dt = - \int_{0}^{T}\!\!\! \int_0^1\tilde{z}v_{t}dxdt - \int_{0}^{T}\!\!\! \int_0^1\tilde{z}(a(x)v_{x})_{x}dxdt}\\
\noalign{\smallskip}
{\displaystyle - \int_{0}^{T}\!\!\! \int_\omega\tilde{v}vdxdt =- \int_{0}^{T}\!\!\! \int_0^1F_{0}\tilde{z}dxdt + \int_{0}^{T}\!\!\! \int_0^1\beta(x)F_{1}\tilde{z}_{x}dxdt- \int_{0}^{T}\!\!\! \int_\omega\tilde{v}vdxdt \leq}
\\ 
\end{array}
\]
\[
\begin{array}{l}
{\displaystyle \biggl(  \int_{0}^{T}\!\!\! \int_{0}^{1} e^{-2s\varphi(x,t)} F_{0}^{2} dxdt  
  + \int_{0}^{T}\!\!\! \int_0^1 s^{2} \theta^{2} \frac{\beta^{2}(x)}{a(x)}
e^{-2s\varphi(x,t)}F_{1}^{2}dxdt  + \int_{0}^{T}\!\!\! \int_\omega
e^{-2s\varphi(x,t)} v^{2}dxdt\biggr)^{1/2} \times}\\
\noalign{\smallskip}
  {\displaystyle \biggl(  \int_{0}^{T}\!\!\! \int_{0}^{1} e^{2s\varphi(x,t)}
{\tilde{z}}^{2} dxdt + \int_{0}^{T}\!\!\! \int_0^1 s^{-2} \theta^{-2}
a(x) e^{2s\varphi(x,t)}{\tilde{z}_{x}}^{2}dxdt
 + \int_{0}^{T}
\int_{\omega} e^{2s\varphi(x,t)}\tilde{v}^{2}dxdt\biggr)^{1/2}.}
\end{array}
\]
Considering \eqref{auxcon2.8}, we obtain
\begin{equation}\label{paso2.2}
\begin{array}{ll}
{\displaystyle  \int_{0}^{T}\!\!\! \int_0^1 s \theta a(x)
e^{-2s\varphi(x,t)} v_{x}^{2} dx dt \leq C
\biggl(  \int_{0}^{T}\!\!\! \int_{0}^{1} e^{-2s\varphi(x,t)} F_{0}^{2} dxdt} \\
\noalign{\smallskip}
 {\displaystyle + \int_{0}^{T}\!\!\! \int_0^1 s^{2} \theta^{3} \frac{\beta^{2}(x)}{a(x)}
e^{-2s\varphi(x,t)}F_{1}^{2}dxdt  + \int_{0}^{T}\!\!\! \int_\omega
e^{-2s\varphi(x,t)} v^{2}dxdt\biggr).}
\end{array}
\end{equation}
combined with \eqref{paso2.1}, this implies \eqref{e.Carleman}, completing the proof.\end{proof} 
\section{Null controllability of semilinear degenerate parabolic equations}
In this section we prove our main result Theorem \ref{T1}. First, as consequence of Carleman inequality, we give a sketch of the proof  Theorem \ref{T2}. As a consequence of this result, we state the null controllability of linear  degenerate parabolic equations but we do not give its prove since nowadays it is classical. See e.g. \cite{Zab},\cite{FI}, \cite{yamamoto} for classical results. Finally we conclude the section with the non linear result.  

 \begin{proof}[Sketch of the proof of Theorem \ref{T2}]
 To prove Theorem \ref{T2} we proceed as in \cite{ACF}, pp.192-195. That is, 
we multiply $v_t +  (a(x)v_{x})_{x}  - b(x,t)v + (\beta(x) c(x,t)v)_{x}   =0$ by  $v$ and we integrate over $(0,1)$. We got that, 
\[ \begin{array}{ll}
0= {\displaystyle \int_{0}^ {1} [ v_t +  (a(x)v_{x})_{x}  - b(x,t)v + (\beta(x)c(x,t)v)_{x}]v dx}\\
={\displaystyle \; \frac{d}{dt}\int_{0}^ {1} v^{2} dx - \int_{0}^ {1} a(x) v_{x}^{2} dx - \int_{0}^ {1} b(x,t) v^{2} dx - \int_{0}^ {1} \beta(x) c(x,t) vv_{x} dx.}
\end{array} \]

Observe that, for  $x\in [0, 1]$, $a(x) \geq a(1)C_{\beta} \beta^{2}(x)$ for some constant$C_{\beta}>0$ as a consequence of~\eqref{condicion.beta}. Then,
\[ \begin{array}{ll}
 {\displaystyle \int_{0}^ {1} a(x) v_{x}^{2} dx = \frac{d}{dt}\int_{0}^ {1} v^{2} dx - \int_{0}^ {1} b(x,t) v^{2} dx
  - \int_{0}^ {1} (2 a(1)C_{\beta} \beta^{2}(x))^{\frac{1}{2}} \frac{c(x,t)}{[2a(1)C_{\beta}]^{\frac{1}{2}}} vv_{x} dx}\\
  \leq \; {\displaystyle \frac{d}{dt}\int_{0}^ {1} v^{2} dx + ||b||_{\infty}\int_{0}^ {1}  v^{2} dx
  +\frac{||c||_{\infty}^{2}}{4a(1)C_{\beta}} \int_{0}^ {1}  v^{2} dx + \int_{0}^ {1} a(1)C_{\beta} \beta^{2}(x) v_{x}^{2} dx.}
\end{array} \]
As\'{\i},
\[ 0 \leq 2 \int_{0}^ {1} (a(x) -a(1)C_{\beta}\beta^{2}(x)) v_{x}^{2} dx \leq \frac{d}{dt} \int_{0}^ {1}v^{2} dx + \biggl(2||b||_{\infty} + \frac{||c||_{\infty}^{2}}{2a(1)C_{\beta}} \biggr) \int_{0}^ {1}v^{2}dx,\]
and
\[ 0 \leq e^{(2||b||_{\infty} + \frac{||c||_{\infty}^{2}}{2a(1)C_{\beta}})t} \biggl[ \frac{d}{dt} \int_{0}^ {1}v^{2} dx + \biggl(2||b||_{\infty} + \frac{||c||_{\infty}^{2}}{2a(1)C_{\beta}} \biggr) \int_{0}^ {1}v^{2}dx\biggr].\]
Therefore, for every $t \in [0,T]$,
\[ 0 \leq \frac{d\biggl[e^{(2||b||_{\infty} + \frac{\|c\|_{\infty}^{2}}{2a(1)C_{\beta}})t}  \int_{0}^ {1}v^{2}(x,t) dx \biggr]}{dt};\]
That means, that for a constant  $C>0$,
\[ \int_{0}^{1} v^{2}(x,0)dx \leq C \int_{0}^{1} v^{2}(x,t)dt, \quad \forall\; t \in [0,T]. \] 

As in \cite{ACF} it is needed to take two cases,  $K\neq 1$ and  $K=1$, to get\[ \int_{0}^{1} v^{2}(x,t)dx  \leq C\int_{0}^{1} a(x) v_{x}^{2}(x,t) dx \quad \forall \; t \in [0,T].\]
Then, for every $t \in [0,T]$
\[ \int_{0}^{1} v^{2}(x,0)dx  \leq C\int_{0}^{1} a(x) v_{x}^{2}(x,t) dx.\]
As a consequence, integrating over  $[T/4, 3T/4]$ and using Lemma~\ref{ineq.aux.Carleman}, we have   
\[
\begin{array}{ll}
{\displaystyle \int_{0}^{1} v^{2}(x,0)dx} & \leq  {\displaystyle C \int_{T/4}^{3T/4}\!\!\! \int_{0}^{1} a(x) v_{x}^{2}(x,t) dxdt}\\
& \leq  C {\displaystyle \int_{T/4}^{3T/4}\!\!\! \int_{0}^{1} s\theta e^{2s\varphi(x,t)} a(x) v_{x}^{2}(x,t) dxdt}\\
 &\leq  C{\displaystyle  \int_{0}^{T}\!\!\! \int_\omega  v^{2} dxdt.}
\end{array}
\]
This concludes the proof of Theorem~\ref{T2}.
\end{proof}
\medskip

Observe that  the observability inequality  implies the  null controllability of system:
\begin{equation}\label{e1-1-1}
    \left\{ \begin{array}{ll} y_t-  (a(x)y_{x})_{x}  + b(x,t)y + \beta(x)c(x,t)y_{x}   =h1_{\omega} & \mbox{in}\;\;Q=(0,1)\times (0,T) , \\
      y(0,t)=y(1,t)=0   & t \in (0,T),\\
    y(x,0)= y_{0}(x), & \mbox{in} \;\; (0,1).
\end{array} \right.
\end{equation}
 In fact, by a minimization argument, it can be proven that

\begin{theorem}\label{tcontrol}
Given $T>0$ and $y_{0} \in L^2(0,1)$, it exists $h \in
L^2(\omega \times (0,T))$ such that the solution   $y$ corresponding to~\eqref{e1-1-1} satisfies
\[
y(x, T) = 0 \;\;\mbox{for almost every}\;\; x \in [0,1].
\] Moreover, it exists a positive constant $C$, only depending on  $T$, such that
\[ \int_{0}^T \!\!\!\int_{\omega}|h|^2 dx dt \leq C \int_{0}^1 y_{0}^2(x)
dx.
\]
\end{theorem}

\begin{proof}[Proof of Theorem \ref{T1}]

 Let $y_{0} \in H_{a}^{1}(0,1)$. Let us assume that 
\[ g(x,t;\cdot), G (x,t; \cdot ) \in C^0(\R^{2})\quad \forall (x,t) \in Q.\]

We define  $Z = L^{2}(0,T; H_{a}^{1}(0,1))$. For every  $z \in Z$, we consider the linear system\begin{equation}\label{e-semi2}
    \left\{ \begin{array}{ll} y_t -  (a(x)y_{x})_{x} + g(x,t; z,z_{x})y + \beta(x)\biggl[{\displaystyle \frac{G(x,t;z,z_{x})}{\beta(x)}}\biggr]y_{x}    = h1_{\omega} & \mbox{en}\;\;Q , \\
    \noalign{\smallskip}
    y(1,t)=0  \quad 
    \mbox{ y } \left\{ \begin{array}{ll} y(0,t)=0 & \mbox{for (WDP)},\\
    (ay_{x})(0,t)=0 &\mbox{for (SDP)},
    \end{array}\right. & t \in (0,T),\\
    y(x,0)= y_{0}(x), & \mbox{in} \;\; (0,1),
\end{array} \right.
\end{equation}
We associate to  $z$ a family of controls $U(z)\subset L^{2}$ that drives the corresponding solution to \eqref{e-semi2} to zero.  Observe that ~\eqref{e-semi2} is of the form~\eqref{e1-1} with
\[ 
\left\{\begin{array}{ll} b=b_{z}= g(x,t; z,z_{x}) \in L^{\infty}(Q)\\
c=c_{z} ={\displaystyle \frac{G(x,t;z,z_{x})}{\beta(x)}} \in L^{\infty}(Q).
\end{array}\right.
\]
From Theorem~\ref{tcontrol}, we deduce the existence of a  control $\widehat{h}_{z} \in L^{2}(\omega \times (0,T))$ such that the solution to~\eqref{e-semi2} with $h = \widehat{h}_{z}$ satisfies
\[ \widehat{y}_{z} (x,T)= 0 \quad \mbox{in} \; (0,1)\]
and, furthermore
\begin{equation}\label{estimacion.h1}
\|\widehat{h}_{z}\|_{L^{2}(\omega \times (0,T))} \leq C \|y_{0}\|_{L^{2}(0,1)}.
\end{equation}
On the other hand, from Theorem~\ref{t.regularidad}, we obtain that
\[ \widehat{y}_{z} \in L^{2}(0,T; H_{a}^{1}(0,1)) \]
and
\begin{equation}\label{estimacion.y1}
\|\widehat{y}_{z} \|_{L^{2}(0,T; H_{a}^{1}(0,1))} \leq C (\|y_{0}\|_{H_{a}^{1}(0,1)} + \|\widehat{h}_{z} \|_{L^{2}(\omega \times (0,T))}).
\end{equation}
Estimates~\eqref{estimacion.h1} and~\eqref{estimacion.y1} can be writen as
\begin{equation}\label{estimacion.h2}
\|\widehat{h}_{z}\|_{L^{2}(\omega \times (0,T))} \leq C \|y_{0}\|_{L^{2}(0,1)}
\end{equation}
and
\begin{equation}\label{estimacion.y2}
\|\widehat{y}_{z} \|_{Z} \leq C \|y_{0}\|_{H_{a}^{1}(0,1)} .
\end{equation}
\medskip

Given $h \in L^{2}(\omega \times (0,T))$, let $y_{h} \in Z$ be the solution to~\eqref{e-semi2} in $Q$ with righthand side $h$ (to simplify notation we omit the dependence on  $z$). For every  $z \in Z$ we define
\[ U(z)= \{ h \in L^{2}(\omega \times (0,T)) : \; y_{h}(T)= 0, \quad \|h\|_{L^{2}(\omega \times (0,T))} \leq C\|y_{0}\|_{L^{2}(0,1)} \}\]
and
\[ \Lambda(z) =\{y_{h} : \; h \in U(z), \quad \|y_{h}|\|_{Z} \leq C \|y_{0}\|_{H_{a}^{1}(0,1)}\}.\]

In this way we introduced a multivalued application  
\[ z \mapsto \Lambda(z). \]

We will show that this application has a fixed point $y$. Of course, this will imply the existence of   $h \in L^{2}(\omega \times (0,T))$ such that~\eqref{e1}  has a solution that satisfies~\eqref{null}.

To this aim we use a version of  Kakutani's fixed point due to   Fau and Glicksberg (see~\cite{FaGl}) that can be applied to $\Lambda$. Firstly, from~\eqref{estimacion.h2} and~\eqref{estimacion.y2} we deduce that $\Lambda(z)$ is, for every $z \in Z$, a non empty set. Secondly, it is easy to check that  $\Lambda(z)$ is a uniformly bounded set, closed and convex in  $Z$. The regularity assumption on $y_0$ and Theorem~\ref{t.regularidad}, implies 
\[y \in  \mathcal{C}^{0}([0,T];H_{a}^1(0,1)) \cap L^2(0,T;
H_{a}^{2}(0,1))\cap H^{1}(0,T; L^{2}(0,1)),\]
and it exists a constant $C_{T}$ such that
\begin{equation}\label{regularidad-simon}
\begin{array}{ll}
 \|y\|_{L^2(0,T;H_{a}^{2}(0,1))} + \|y_{t}\|_{L^{1}(Q)}
\leq \;C_{T} ||y_{0}||_{H_{a}^1(0,1)} .
\end{array}
\end{equation}
(where $C_{T}$ is independent of $z$) for every $y \in \Lambda(z)$. Furthermore,  $H_{a}^{2}(0,1)$ is compactly imbebed in $H_{a}^{1}(0,1)$ (see, e.g.~\cite{ACF}). We can conclude that it exists a compact set $K \subset Z$ such that
\begin{equation}\label{simon}
\Lambda(z) \subset K \quad \forall \; z \in Z \;\mbox{(see~\cite{Simon}).}
\end{equation}
 
 Now we prove that  $z \mapsto \Lambda(z)$ is upper hemicontinuous, i.e.  the real function \[ 
  z \in Z \mapsto   \sup_{y \in \Lambda(z)} \langle\mu,y\rangle
 \]
is upper semi-continuous for every $\mu \in Z'$. In other words, we will check that
 \[ B_{\alpha, \mu} =\{ z \in Z: \sup_{y \in \Lambda(z)} \langle\mu,y\rangle \geq \alpha\} \]
is a closed set of $Z$ for every $\alpha \in \R$. To this end, take $\{ z_{n} \}$  a sequence in $B_{\alpha, \mu}$ such that \[ z_{n} \rightarrow z \quad \mbox{en} \quad Z.\]
We want to show that $z \in B_{\alpha, \mu}$. We know, that it exists a subsequence $\{ z_{n_{k}} \}$ such that
 \[ 
 z_{n_{k}}(x,t) \rightarrow z(x,t), \;\; \mbox{almost everywhere in}\;\; Q,
\]
and
\[\sqrt{a(x)}z_{n_{k},x}(x,t) \rightarrow \sqrt{a(x)}z_{x}(x,t),\;\; \mbox{almost everywhere in} \;\; Q.
\]
Then, from the continuity assumptions on  $g$ and $G$,  we got
 \[ 
g(x,t; z_{n_{k}},z_{n_{k},x}) \rightarrow g(x,t; z, z_{x}) \;\; \mbox{in}\;\; L^{\infty}(Q),
\]
and
\[ \frac{G(x,t; z_{n_{k}},z_{n_{k},x})}{\beta(x)} \rightarrow \frac{G(x,t; z, z_{x})}{\beta(x)} \;\; \mbox{in} \;\; L^{\infty}(Q).
\]

On the other hand, since all the sets $\Lambda(z_{n})$ are compact and satisfy~\eqref{simon}, we deduce that
\begin{equation}\label{cerrado}
 \alpha \leq  \sup_{y \in \Lambda(z_{n_{k}})} \langle \mu,y \rangle = \langle \mu,y_{n_{k}} \rangle
 \end{equation}
for some $y_{n_{k}} \in \Lambda(z_{n_{k}})$. From the definition of  $\Lambda(z_{n_{k}})$ and $U(z_{n_{k}})$, it exists $h_{n_{k}} \in L^{2}(\omega \times (0,T))$ such that
\[ y_{t,n_{k}} -  (a(x)y_{x,n_{k}})_{x} + g(x,t; z_{n_{k}},z_{x,n_{k}})y_{n_{k}} + \beta(x)\biggl[\frac{G(x,t;z_{n_{k}},z_{x,n_{k}})}{\beta(x)}\biggr]y_{x,n_{k}}    = h_{n_{k}}1_{\omega} \quad \mbox{in}\;\;Q.\]
Moreover,
\[
\|h_{n_{k}}\|_{L^{2}(\omega \times (0,T))} \leq C \|y_{0}\|_{L^{2}(0,1)}
\]
and
\[
\|y_{n_{k}} \|_{Z} \leq C \|y_{0}\|_{H_{a}^{1}(0,1)} .
\]
where $y_{n_{k}}$ (resp. $h_{n_{k}}$) is uniformly bounded in  $Z$ (resp. $L^{2}(\omega \times (0,T))$). Then , 
\[ y_{n_{k}} \rightarrow \widehat{y} \;\; \mbox{ strongly in }\;\; Z\]
and 
\[ h_{n_{k}} \rightarrow \widehat{h} \;\; \mbox{weakly in }\;\; L^{2}(\omega \times (0,T)).\]

It is not difficult to show that
\[
    \left\{ \begin{array}{ll} \widehat{y}_t -  (a(x)\widehat{y}_{x})_{x} + g(x,t; z,z_{x})\widehat{y} + \beta(x)\biggl[{\displaystyle \frac{G(x,t;z,z_{x})}{\beta(x)}}\biggr]\widehat{y}_{x}    = \widehat{h}1_{\omega} & \mbox{in}\;\;Q , \\
    \noalign{\smallskip}
    \widehat{y}(1,t)=0  \quad 
    \mbox{ and} \left\{ \begin{array}{ll} \widehat{y}(0,t)=0 & \mbox{for (WDP)},\\
    (a\widehat{y}_{x})(0,t)=0 &\mbox{for (SDP)},
    \end{array}\right. & t \in (0,T),\\
    \widehat{y}(x,0)= y_{0}(x), \quad \widehat{y}(x,T)=0 & \mbox{in} \;\; (0,1),
\end{array} \right.
\]
in the sense of distributions; that is $\widehat{v} \in U(z)$ and $\widehat{y} \in \Lambda(z)$. In consequence, we can take the limit in ~\eqref{cerrado} to deduce that
\[ \alpha \leq \langle \mu, \widehat{y} \rangle \leq \sup_{y \in \Lambda(z)} \langle \mu, y\rangle.\]
that means that, $z \in B_{\alpha,\mu}$. We obtain that  $z \mapsto \Lambda(z)$ is upper  hemi-continuous.

We assume now that  $g(x,t;\cdot)$ and  $G(x,t;\cdot)$ belong to $L^{\infty}(\R^{2})$. We introduce  the function $\rho(x,t;\cdot) \in C_{c}^{\infty}(\R^{2})$ such that $\rho(x,t;\cdot) \geq 0$ in $\R^{2}$, $\mbox{supp} \;\rho(x,t;\cdot) \subset \overline{B}(0,1)$ and
\[ \int \int_{\R^{2}} \rho(x,t;s,p)dsdp =1.\]
for every $(x,t) \in Q$. We consider the functions $\rho_{n}, g_{n}$ y $G_{n}$ ($n \geq 1$), with
\[ \rho_{n}(x,t; s,p) = \frac{1}{n^{2}}\rho(x,t;ns,np) \quad \forall \; (s,p) \in \R^{2},\]
and
\[ g_{n} = \rho_{n} \ast g, \quad G_{n} = \rho_{n} \ast G.\]
Then, it is not difficult to see that $g_{n}$ and $G_{n}$ satisfy:
\begin{enumerate}
\item[1.] $g_{n}(x,t;\cdot), G_{n}(x,t;\cdot) \in L^{\infty}(\R^{2})$   ($n\geq 1$).
\item[2.]  $g_{n}(x,t;\cdot) \rightarrow g(x,t;\cdot)$ and $G_{n}(x,t;\cdot) \rightarrow G(x,t;\cdot)$ uniformly in $\R^{2} $ for  $(x,t) \in Q$.
\end{enumerate}
\medskip

For every  $n$, we obtain a control $h_{n}\in L^{2}(\omega \times (0,T))$such that
\begin{equation}\label{casogeneral}
    \left\{ \begin{array}{ll} y_{t,n} -  (a(x)y_{x,n})_{x} + f(x,t,y_{n},y_{x,n}) = h_{n}1_{\omega} & \mbox{in}\;\;Q , \\
    \noalign{\smallskip}
    y_{n}(1,t)=0  \quad 
    \mbox{ and } \left\{ \begin{array}{ll} y_{n}(0,t)=0 & \mbox{for (WDP)},\\
    (ay_{x,n})(0,t)=0 &\mbox{for (SDP)},
    \end{array}\right. & t \in (0,T),\\
    y_{n}(x,0)= y_{0}(x), & \mbox{in} \;\; (0,1),
\end{array} \right.
\end{equation}
has a least one solution  $y_{n} \in Z$ that satisfies
\[ y_{n}(x,T) =0 \quad \mbox{en}  \quad (0,1),\]

\[
\|h_{n}\|_{L^{2}(\omega \times (0,T))} \leq C 
\quad \mbox{and} \quad \|y_{n} \|_{Z} \leq C \quad \forall\; n \geq 1 .
\]
We have that $y_{n} \in K$ for every $n\geq 1$, with $K$ a compact subset of $Z$. Therefore, we can assume that at least for a subsequence
\[ y_{n} \to y \;\; \mbox{ strongly in}\;\; Z,\]
\[ h_{n} \rightharpoonup h \;\; \mbox{weakly in}\;\; L^{2}(\omega \times (0,T)).\]
Passing to the limit in ~\eqref{casogeneral}, we obtain a control $h \in L^{2}(\omega \times (0,T))$ such that it exists a control $h\in L^2(Q)$ such that the corresponding solution a solution $y$  to~\eqref{e1} that satisfies~\eqref{null}. The regularizing effects of the degenerate parabolic equation allows to prove the result when $y_{0} \in L^{2}(0,1)$. In fact in this situation take $h=0$ in a time interval $(0, t_0)$ with $t_0<T$ and then control the equation in $(t_0, T)$,  with initial datum $y(t_0)$.\end{proof}

\end{document}